\documentclass[12pt]{amsart}
\title[An Exercise (?) in Fourier Analysis]{An Exercise (?) in Fourier Analysis on the Heisenberg Group}
\author[Bump]{Daniel Bump}\thanks{The first three authors  can be contacted at: Department of Mathematics, Stanford University, 450 Serra Mall, Bldg. 380, Stanford,
CA 94305-2125, USA}\author[Diaconis]{Persi Diaconis}\author[Hicks]{ Angela Hicks}\author[Miclo]{ Laurent Miclo} \author[Widom]{ Harold Widom}\thanks{The fourth author can be contacted at:  Institut de Math\'{e}matiques de Toulouse, Universit\'{e} Paul Sabatier, 118, route de Narbonne, F-31062 Toulouse {Cedex 9}, France.} \thanks{The fifth author can be contacted at: UC Santa Cruz, Department of Mathematics, Santa Cruz, CA 95064, USA} 
\newcommand{\nZ}{{\mathbb{Z}}/{n\mathbb{Z}}}
\usepackage{enumerate}
\newtheorem{proposition}{Proposition}

\newtheorem{theorem}{Theorem}
\theoremstyle{definition}
\newtheorem*{remark}{Remark}
\newtheorem{example}{Example}
\newtheorem*{acknowledgment}{Acknowledgment}
\usepackage{amsmath,amssymb,tikz,accents}
\usetikzlibrary{matrix, arrows,positioning}
\begin{document}

\begin{abstract}
Let $H(n)$ be the group of $3\times 3$ uni-uppertriangular matrices with entries in ${\mathbb{Z}}/{n\mathbb{Z}}$, the integers mod $n$.  We show that the simple random walk converges to the uniform distribution in order $n^2$ steps.  The argument uses Fourier analysis and is surprisingly challenging.  It introduces novel techniques for bounding the spectrum which are useful for a variety of walks on a variety of groups.
\end{abstract}
\maketitle
\begin{center}\today\end{center}
\section{Introduction}\label{PS1}  For a positive integer $n$, let $H(n)$ be the ``Heisenberg group mod $n$," the group of $3\times 3$ uni-uppertriangular matrices with entries in $\nZ$:
$$\begin{pmatrix}1&x&z\\0&1&y\\0&0&1\end{pmatrix}\hspace{1cm} x,y,z\in \nZ.
$$

Write such a matrix as $(x,y,z)$, so $$(x,y,z)(x',y',z')=(x+x',y+y',z+z'+xy').$$  A minimal symmetric generating set is \begin{align}S=\{(1,0,0),(-1,0,0),(0,1,0),(0,-1,0)\}\end{align}\label{P1.1}
A random walk proceeds, starting at the identity $(0,0,0)$, by repeatedly choosing an element from $S$ with probability $\frac{1}{4}$ and multiplying.  More formally, define
\begin{align}\label{P1.2}Q(g)=\begin{cases}\frac{1}{4}& g\in S\\0&\text{ otherwise}\end{cases}.
\end{align}
For $g\in H(n)$, $$Q*Q(g)=\sum_h Q(h) Q(gh^{-1}), \,\,\, Q^{*k}(g)=\sum_h Q(h)Q^{*(k-1)}(gh^{-1}).$$
If $n$ is odd, high convolution powers of $Q$ converge to the uniform distribution $U(g)=\frac{1}{n^3}$, $Q^{*k}(g)\rightarrow U(g)$.  Measure the speed of convergence by the total variation distance
$$\|Q^{*k}-U\|=\max_{A\subseteq H}|Q^{*k}(A)-U(A)|=\frac{1}{2}\sum_{g\in H} |Q^{*k}(g)-U(g)|.$$
One purpose of this paper is to give a new proof of the following theorem.
\begin{theorem}\label{PT1}
For $Q$ defined by (\ref{P1.2}) on $H(n)$, there are universal positive constants $A$ and $C$ such that 
$$Ae^{\frac{-2\pi^2k}{n^2}}\leq\|Q^{*k}-U\|_{TV}\leq Ce^{\frac{-2\pi^2k}{n^2}}$$
for all odd $n$ and all $k$.
\end{theorem}
We say ``order $n^2$ steps are necessary and sufficient for convergence."  Theorem \ref{PT1} is proved by Fourier analysis on $H(n)$.  This natural approach is known to be challenging.  (See Section \ref{PS2}.)  As explained in Section \ref{PS3}, the Fourier proof involves bounding the spectrum of the $n\times n$ matrices (where the unmarked entries are taken to be zero):
\begin{equation}\label{ourmatrix}M(\xi)=\frac{1}{4}\begin{tikzpicture}[baseline=(current bounding box.center)]
\matrix (m) [matrix of math nodes,nodes in empty cells,right delimiter={)},left delimiter={(} ]{
2  & 1 &  & &&   & 1  \\
 1& & & & &&   \\
  & & & & &&   \\
  & & 1&\phantom{1} &1& &    \\
   & & & & &&   \\
  & & & & &&1  \\
1 & & && &  1&\phantom{1} \\
} ;
\draw[loosely dotted,thick] (m-1-1)-- (m-4-4);
\draw[loosely dotted,thick] (m-2-1)-- (m-4-3);
\draw[loosely dotted,thick] (m-1-2)-- (m-4-5);
\draw[loosely dotted,thick] (m-4-4)-- (m-7-7.center);
\draw[loosely dotted,thick] (m-4-3)-- (m-7-6);
\draw[loosely dotted,thick] (m-4-5)-- (m-6-7);
\node[align=center] (cosgen) at (4.6cm,1cm) [below=3mm]{$2\cos\left(\frac{2\pi \xi j}{m}\right)$,\,\, \tiny{$0\leq j \leq m-1$}};
\path[thick, bend left=45, <-] 
 (m-4-4.center) edge (cosgen.west);
\end{tikzpicture}\end{equation}
These same matrices come up in a variety of solid state physics and ergodic theory problems as `Harper's Operator,' `Hofstader's Butterfly,' and the `Ten Martini's Problem' discussed in Section \ref{PS2}.  They further occur in the engineering literature on the fast Fourier transform and for other random walks on the meta-abelian groups.

Section \ref{PS2} is a literature review, including the parallel problem of a random walk on $H(\mathbb{R})$ and $H(\mathbb{Z})$.  It describes the physics problems, the connections with the fast Fourier transform, and connections to Levy's area.  Section \ref{PS3} gives background on Fourier analysis on $H(n)$.  Theorem \ref{PT1} is proved there assuming our eigenvalue bounds (which are proved later).  It also gives a variety of other walks on other groups where matrices $M(\xi)$ (and thus our bounds) appear.

Eigenvalue bounds are given in Section \ref{PS4}.  Two new techniques are developed.  The first uses geometric path methods, developed for Dirichlet eigenvalues of stochastic matrices to give upper bounds for the largest eigenvalue of $M(\xi)$; note that some entries of $M(\xi)$ are negative and the rows don't sum to a constant.  The second is a new technique for lower bounds on smallest eigenvalues.  The bounds are rough (but more than adequate for the task) and work generally.
\begin{acknowledgment}  We have benefited from useful comments and technical help from Laurent Saloff-Coste, Marty Isaacs, Je-Way Zhou,  Steven Kivelson, Steve Evans, Mehrdad Shahshahani, Thomas Strohmer, Yi Zhang, Akash Maharaj, and Arun Ram. Also, the first, second, third, and fifth authors would like to acknowledge partial support from NSF grants DMS 1001079, DMS 08-04324, DMS 1303761, and DMS 1400248 (respectively).  The remaining author would like to acknowledge partial support from ANR grant number ANR-12-BS01-0019.
\end{acknowledgment}
\section{Background and References}\label{PS2}
For background on Markov chains and random walks on groups, a good introduction is \cite{LPW}.  The Fourier approach is developed in \cite{DiaconisB}.  A variety of analytic approaches are described in \cite{SCB} and \cite{SCA}. The Heisenberg group is ubiquitous in mathematics and physics through its connections to the quantum harmonic oscillator \cite{Cartier}, Harmonic analysis \cite{Howe}, theta functions (\cite{Igusa}, \cite{LV}), and number theory {\cite{Bump}.  Analysis and geometry on the Heisenberg group have recently been important in finding counterexamples to the Goemans-Linial conjecture in complexity theory \cite{LN}.  When $p$ is prime, the groups $H(p)$ are building blocks of the extra special $p$-groups (groups with center, commutator and derived group equal to $\mathbb{Z}/p\mathbb{Z}$).  See \cite{Huppert} and \cite{Suzuki}.

The literature on random walks on the Heisenberg group is reviewed in \ref{PS2.1}, the connection to fast Fourier transform is treated in \ref{PS2.2} the solid state connection is reviewed in \ref{PS2.3}, and the connection to Levy's area is given in \ref{PS2.3}.  Section \ref{PS2.5} describes two companion papers.  The first generalizes to other groups; the second gives precise asymptotics for the top and bottom eigenvalues of $M(\xi)$.

\subsection{Random Walk on $H(n)$}\label{PS2.1}  The probabilistic study of a random walk on $H(n)$ was started by Zack (\cite{Zack}), who gave an interpretation of the walk in terms of random number generation.  Theorem \ref{PT1} was first proved in joint work with Saloff-Coste (see \cite{DSCC}, \cite{DSCB}, \cite{DSCA}) using `geometric theory.'  The first proof used Nash inequalities, the second proof uses moderate growth, showing that a similar proof holds for any sequence of finite nilpotent groups with bounded class and number of generators: $(\text{diameter})^2$ steps are necessary and sufficient for convergence. (The diameter of $H(n)$ with generating set $S$ is of order $n$.)  The third proof of Theorem \ref{PT1} lifts to a random walk on a free nilpotent group (here $H(\mathbb{Z})$).  Then results of Herbish-Saloff-Coste give bounds on the decay rate of convolution powers for the lifted walk.  These are transferred down to the finite quotient using Harnack inequalities.  The beautiful work of Alexopoulos (\cite{Alexopoulos}) should allow similar results for less symmetric walks on quotients of groups with polynomial growth. 

  A fourth proof appears in \cite{DiaconisA}.  This completes the generating set $S$ to a union of conjugacy classes.  Character theory is used to bound the conjugacy walk.  Finally, comparison theory is used to go from the conjugacy walk to the walk based on S. Richard  Stong (\cite{StongC}, \cite{StongB}, \cite{StongA}) has results which specialize to three yet different proofs.  His approach, using `twisted random walks' works for other classes of groups, such as the Burnside groups $B(n,3)$.  It deserves to be developed further.
  
    A further set of approaches/results is worth mentioning.    A simple random walk on $n\times n$ uni-uppertriangular matrices with coefficients mod $p$ has also been studied in a series of papers.  In this work, the emphasis is on large $n$ and fixed $p$ behavior but some of the results specialize well to $n=3$.  We refer to the definitive paper by Peres and Sly \cite{PS} which surveys earlier literature.
    
    Curiously \emph{none} of the papers cited above carries out the program attempted here\textemdash direct use of Fourier analysis.  All we can say is, having pushed it through, we now see why.

    There has also been work on a random walk on the Heisenberg group $H(\mathbb{R})$.  This can be developed as a special case of a random walk on Lie groups\textemdash see Breuillard \cite{Breuillard2} for a useful survey.  The paper \cite {Breuillard1} focuses on $H(\mathbb{R})$ and proves a local limit theorem using Fourier analysis.  Of course, this is close in spirit to the present effort.  Alexopoulos \cite{Alexopoulos} had remarkable decay rates for general discrete groups of polynomial growth.
    
    A natural problem (open to our knowledge) interpolating between the finite and and infinite case would be to study the simple random walk on $H(\mathbb{Z})$ with generating set $S$ of (\ref{PS1}).  Although Breuillard allows discrete measures on $H(\mathbb{R})$, he must rule out measures supported on subgroups.  Indeed, as far as we know, the irreducible unitary representations of $H(\mathbb{Z})$  have not been classified (and seem to give rise to a `wild' problem).
    
    If $Q$ is the measure supported on $S$ (considered on $H(\mathbb{Z})$), the associated random walk is transient.  It is natural to study the decay rate of $Q^{*k}$.  In particular we conjecture that 
   \begin{align}Q^{*k}(\text{id})\sim\frac{c}{k^2} \text{  with } c=\frac{4\Gamma^2\left(\frac{1}{4}\right)}{\pi^{2}}.\label{Aeq1}\end{align}
 A justification for the conjectured constant appears in Section \ref{PS2.4} below.
   
   Notice that (\ref{Aeq1}) could be studied by considering $Q$ from (\ref{P1.2}) on $H(n)$ for $n>>k^2$ (when there is no possibility of the walk wrapping around.)  Indeed $H(n\mathbb{Z})$ is a normal subgroup of $H(\mathbb{Z})$ and all of the irreducible unitary representations of $H(\mathbb{Z}/n\mathbb{Z})$ described in \ref{PS3.2} give rise to irreducible unitary representations of $H(\mathbb{Z})$ via
   $$H(\mathbb{Z})\rightarrow H(\mathbb{Z})/H(n\mathbb{Z})=H(\mathbb{Z}/n\mathbb{Z})\rightarrow GL(V)$$
where $V$ is the vector space underlying the representation of $H(\mathbb{Z}/n\mathbb{Z})$.   Those quotients are all of the finite dimensional simple representations of $H(\mathbb{Z})$.  (See \cite{NM}.)  The usual infinite dimensional irreducible unitary representations of $H(\mathbb{R})$ restricted to $H(\mathbb{Z})$ are also irreducible but alas there are others which do not seem to be known at this writing.

\subsection{The Discrete Fourier Connection} \label{PS2.2} For $f:\mathbb{Z}/n\mathbb{Z}\rightarrow \mathbb{C}$, the discrete Fourier transform (DFT) takes $f\rightarrow \hat{f}$ with $$\hat{f}(k)=\frac{1}{\sqrt{n}}\sum_{\xi=0}^{n-1} f(\xi)e^{\frac{2\pi i \xi k}{n}}.$$  This is a mainstay of scientific computing because of the fast algorithms (FFT).  There is a close connection between the discrete Heisenberg groups $H(n)$ and these transforms.  This is charmingly exposited in Auslander and Tolimieri (\cite{AT}).  Let the matrix associated with the DFT be $\mathcal{F}_n$, with $(\mathcal{F}_n)_{j,k}=\frac{e^{\frac{2\pi i jk}{n}}}{\sqrt{n}}$.  Recall the matrix of the Fourier transform $M(1)$ at (\ref{ourmatrix}).  It is straightforward to see that $\mathcal{F}_n$ commutes with $M(1)$:
$$\mathcal{F}_n^{-1} M(1) \mathcal{F}_n=M(1).$$
   This implies that these Hermitian matrices can be simultaneously diagonalized.  As explained in \cite{DS} and \cite{Mehta}, there is engineering interest in diagonalizing $\mathcal{F}_n$.  We are interested in diagonalizing $M(1)$ and hoped information about $\mathcal{F}_n$ would help.
   
   Alas, ${\mathcal{F}_n}^{4}=\text{id}$, so the eigenvalues are $\pm 1, \pm i$.  It is known that the associated eigenspaces have dimension (within one of) $\frac{n}{4}$.(\cite{DS})  This does not give a useful reduction for our problem.  Going the other way, Dickinson and Steiglitz \cite{DS} looked for matrices commuting with ${\mathcal{F}_n}$ and suggested $M(1)$(!).  Mehta \cite{Mehta} found a basis for eigenfunctions of ${\mathcal{F}_n}$ as `periodicized Hermite functions'.  These are tantalizingly close to the approximate eigenfunctions for $M(1)$ that we find in \cite{BDHMW2}.  We are sure there is more to understand on this front.
   \subsection{Harpers Operator} \label{PS2.3}In trying to understand the band structure of the x-ray diffraction patterns of various solids (particularly metals),a variety of model problems involving Schr\"odinger operators with periodic potentials have been introduced and studied.  Perhaps the most widely studied model is Harper's model on $\ell^2(\mathbb{Z})$ with Hamiltonian $$H_{\theta,\phi}\xi(n)=\xi(n+1)+\xi(n-1)+2\cos(2\pi(n\theta+\phi))\xi(n),$$
   $n\in \mathbb{Z}$, $\theta, \phi\in [0,1]$.  The spectrum of $H_{\theta, \phi}$ is pictured as ``Hofstadter's Butterfly'' with its own Wikipedia page.  With $2\cos$ replaced by $A\cos$, the operator is often called the almost Mathieu operator (with its own Wikipedia page).  For $\theta$ irrational, the norm of $H_{\theta,\phi}$ is known not to depend on $\phi$.  Our operator $M(\xi)$ is a discretized version of $H_{\frac{\xi}{n},0}$ which is often used to get numerical approximations to the spectrum.  It has also been used to prove things, see the beautiful account in Section 3 of B{{\'e}guin-Valette-Zuk \cite{BVZ}.  They give bounds on the largest and smallest eigenvalues of $M(\xi)$ which we are able to sharpen considerably in some cases.
   
   Any periodic Schr\"{o}dinger operator has band structure in its spectrum because of Bloch's Theorem \cite{Griffiths}.  Determining the fine structure of the spectrum of $H_{\theta,\psi}$ has generated a large literature well surveyed in \cite{Last}.  One of the signature problems here was dubbed the ``Ten Martini's problem'' by Barry Simon.  It has recently been completely resolved by work of Last \cite{Last}, and Avila and Jitomirskaya (\cite{AJ}).
   
   There is considerable interest in a variety of similar operators, e.g.\ in $M(\xi)$, replace $\cos\left(\frac{2\pi\theta j}{n}\right)$ with $\cos\left(\frac{2\pi\theta j}{n}\right)+\cos\left(\frac{2\pi\xi j}{n}\right)$.   See Zhang et al \cite{ZMK} and their references.  We have tried to develop our bounds on the spectrum of $M(\xi)$ (in Section \ref{PS4}) so that it applies to these generalizations.
   
   \subsection {Levy Area Heuristics}\label{PS2.4}
   The random walk on $H(\mathbb{Z})$ with generating set $S=\{(1,0,0),(-1,0,0),(0,1,0),(0,-1,0)\}$ leads to interesting probability theory which gives a useful heuristic picture of the behavior on $H(n)$ as well.  This section gives the heuristics; mathematical details will appear in a later publication.
   
    Let the successive steps of the walk be  $$(\epsilon_1,\delta_1,0),(\epsilon_2,\delta_2,0),\dots, (\epsilon_n,\delta_n,0),\dots\in S.$$  By a simple direct computation, $(\epsilon_n,\delta_n,0)\dots(\epsilon_1,\delta_1,0)=(X_n,Y_n,Z_n)$ with 
    $$X_n=\epsilon_1+\dots\epsilon _n,\,\,\, Y_n=\delta_1+\dots+\delta_n,$$
    $$Z_n=\epsilon_2\delta_1+\epsilon_3(\delta_1+\delta_2)+\dots+\epsilon_{n}(\delta_1+\dots+\delta_{n-1}).$$
    Thus the $X_n$ and $Y_n$ coordinates perform simple random walks on $\mathbb{Z}$ $\left(P(\epsilon_i=1)=P(\epsilon_i=-1)=\frac{1}{4}\right)\text{ and }\left(P(\epsilon_i=0)=\frac{1}{2}\right)$.  The Local Central Limit Theorem shows
    $$P\{X_n= j\}\sim \frac{e^{-\frac{1}{2}\left(\frac{j}{\sigma\sqrt{n}}\right)^2}}{\sqrt{2\pi \sigma^2 n}}\,\,\, \text{ with } \sigma^2=\frac{1}{2}.$$
    
    The same result holds for $Y_n$ and indeed $(X_n,Y_n)$ have independent Gaussian approximations.  In particular, $$P\{(X_n,Y_n)=(0,0)\}\sim \frac{1}{\pi n}.$$
    
    Next consider $Z_n$.  This is a martingale and a straightforward application of the martingale central limit theorem shows that $$P\left\{\frac{Z_n}{n}\leq \xi\right\}\rightarrow F(\xi)$$ with $F(\xi)$ the distribution function of $$\frac{1}{2}\int_0^1 B_1(s)dB_2(s)$$
    with $B_i$ standard Brownian motion.  We note that $\int_0^1B_1(t)dB_2(t)-\int_0^1B_2(t)dB_1(t)$ is Levy's area integral, the distribution of the signed area between Brownian motion and its chord at time 1.
    
    In turn, $$\int_0^1 B_1(s)dB_2(s)$$ can be shown to have the density
    $$f(\xi)=\frac{1}{\pi^{\frac{3}{2}}}\Gamma\left(\frac{1}{2}\right)2^{\frac{3}{2}}\Gamma\left(\frac{1}{4}+\frac{i\xi}{2}\right)\Gamma\left(\frac{1}{4}-\frac{i\xi}{2}\right),\,\,\, -\infty<\xi<\infty$$
    
    See Harkness and Harkness \cite{HH} for background and details.
    
    If $\frac{X_n}{\sqrt{n}},\frac{Y_n}{\sqrt{n}},\frac{Z_n}{\sqrt{n}}$ are asymptotically independent, this would suggest that
    $$P\{(X_n,Y_n,Z_n)=(0,0,0)\}\sim \frac{c}{n^2},\,\,\, c=\frac{4\Gamma^2\left(\frac{1}{4}\right)}{\pi^{2}}.$$
    Note that Breuillard \cite{Breuillard1} has a local limit theorem of a similar sort for a non-lattice walk on $H(\mathbb{R})$ with a less well identified constant.  We conjecture that this is our $c$.
    A further heuristic conclusion which this paper makes rigorous:  the walk on $H(\mathbb{Z})$ can be mapped to the walk on $H(m)$ by considering the coordinates mod $m$.  This familiar trick works, for example, for $X_n (\text{mod } m)$, showing that the simple random walk mod $(m)$ requires order $m^2$ steps to get close to uniform.  The same argument works for the joint distribution of $(X_n,Y_n)$.  The proof for these cases requires something like the Berry-Esseen Theorem.  See Diaconis \cite{DiaconisB} for details for $X_n (\text{mod } m)$ and \cite{Alexopoulos} for details on $H(n)$.  The limit theorem  for $\frac{Z_n}{n}$ suggests that order $m$ steps are necessary and sufficient.  This suggest that order $m^2$ steps are necessary and sufficient for randomness on $H(m)$.  This is exactly what Theorem \ref{PT1} proves.
    
    This aspect of `features' of a Markov chain getting to their stationary distribution at different rates is important in practice.  A simulation is often run with just a few features of interest and the relaxation time for the whole state space may be deceptive.  See \cite{ADS} for references to a handful of cases where this can be carefully proved.
    Generalizing, if $U_n(m)$ is the group of $n\times n$ uni-uppertriangular matrices with coefficients mod $m$, let a simple random walk be determined by picking one of the $n-1$ coordinates just above the diagonal and placing $\pm 1 $ there (with the rest zero.)  We conjecture that coordinates in the first diagonal require order $m^2$ step; in the second diagonal, order $m$ steps; in the third diagonal, order $m^{\frac{2}{3}}$ steps; with order $m^{2/k}$ steps for the $k$th diagonal.  Here $n$ is fixed, $1\leq k\leq n-1$ and $m$ large are the cases of interest.
    \subsection{Future Developments} \label{PS2.5}  The present paper has two companion papers.  The first \cite{BDHMW3} abstracts the present results to a larger class of walks and groups.  The second gets much more precise bounds on the extreme eigenvalues of the matrices $M(\xi)$.  Since essentially these matrices come up for the generalizations, present techniques and the refinements can be used to give satisfactory analyses of the generalizations.  
    
    Let $G$ be a finite group with a normal abelian subgroup $A$ and a subgroup $B$ such that $G=BA$, $B\cap A =\mathit{id}$.  Thus $G=B\ltimes A$ is a semi-direct product. Let $a_1,\dots, a_s$ be a symmetric generating set for $A$; $b_1, \dots, b_t$ a symmetric generating set for $B$.   Then $\{a_i,b_j\}$ generate $G$.  Consider the probability measure (with $0<\theta<1$ fixed):
    $${Q}(g)=\begin{cases}\frac{\theta}{s} &\text{ if }g \in \{a_1,\dots,a_s\}\\\frac{1-\theta}{t}&\text{ if }g \in \{b_1,\dots,b_t\}\\0& \text{ otherwise}\end{cases}.$$
    
    The representation theory of semi-direct products is well known. See Serre Chapter 7 (\cite{Serre}).  The irreducible representations are determined by the representations of $B$ and by inducing the characters of $A$ up to $G$.  Let $\chi \in \hat{A}$.  If $\rho_\chi(b,a)$ is the induced representation (a $|B|\times|B|$ matrix),
    $$\hat{{Q}}(\rho_\chi)=\theta P_B+ (1-\theta)D,$$
    with $P_B$ the transition matrix of the $B$-walk (based on a uniform choice of $b_i$) and $D$ a diagonal matrix with $b$th element $$D(b,b)=\frac{1}{s}\sum_{i=1}^s\chi(b^{-1}a_ib).$$  This is in sharp parallel with the present development and we have found that the techniques of the present Section \ref{PS4} suffice for successful analyses.   Examples from \cite{BDHMW3} include:
    \begin{itemize}
    \item The random walk on $H(n)$ generated by $(x(i),0,0)$ and $(0,y(j),0)$ for (perhaps non symmetric) generating sets $1\leq i\leq I$, $1\leq j \leq J$ chosen with some (perhaps non-uniform) probabilities.     
    \item The $d$-dimensional Heisenberg groups $G=(\undertilde{x},\undertilde{y},z)$, $\undertilde{x},\undertilde{y}\in(\mathbb{Z}/n\mathbb{Z})^d$, $z\in (\mathbb{Z}/n\mathbb{Z})$.
    \item The affine groups: $G=(a,b)$, $\theta\in \mathbb{Z}/n\mathbb{Z}$, $a \in (\mathbb{Z}/n\mathbb{Z})^*$

    \end{itemize}
These and further examples are treated in \cite{BDHMW3}.  That paper also develops a useful supercharacter theory which allows the comparison approach \cite{DSCD},\cite{DiaconisA} to be used for this class of examples.
The second companion paper gives sharp bounds on the eigenvalues and eigenvectors of matrices $M(\xi)$, for certian choices of $\xi$ and dimension $m$.  For ease of exposition, work with $M(1)$.  It is shown that the $k$th largest eigenvalue of $M(1)$ is equal to $1-\frac{\mu_k}{2n}+o\left(\frac{1}{n}\right) $ where $\mu_k$ is the $k$th smallest eigenvalue of the harmonic oscillator
$$L=-\frac{1}{2}\frac{d^2}{dx^2}+2\pi^2x^2.$$
 on $(-\infty,\infty)$ (so $\mu_k=(2k+1)\pi)$, $0\leq k <\infty$ fixed.   The corresponding eigenfunctions of $L$ are well known to be Hermite functions; with $a=\pi$,
 \begin{align*}
 \phi_0(x)&=e^{-ax^2}\\
 \phi_1(x)&=-4axe^{-ax^2}\\
 &\vdots\\
 \phi_n(x)&=e^{ax^2}\frac{d^n}{dx^n}(e^{-2ax^2})=e^{ax^2}\frac{d}{dx}e^{-ax^2}\phi_{n-1}(x)
 \end{align*}
 It is also shown that wrapped versions of these are approximate eigenfunctions of $M(1)$.  Numerical comparisons, not reported here, show that these approximate eigenfunctions are very accurate for fixed $k$ and large $n$.
 
 The present technique gives cruder results which are all that we need and are available for all $\xi$.
    \section{Representation Theory} \label{PS3}
    This section develops the tools needed for representation theory:  Fourier inversion, the Plancherel Theorem, representations of $H(n)$, and a formula for the distribution of central elements under convolution powers.  Theorem \ref{PT1} is proved when $n$ is prime assuming the eigenvalue bounds in the following Section \ref{PS4}.  For background, see \cite{Serre} and \cite{DiaconisB}.
    \subsection{Basic Formulas} \label{PS3.1}  Let $G$ be a finite group with $\hat{G}$ the irreducible unitary representations. For ${Q}$ a probability on $G$ and $\rho\in \hat{{G}}$, define the Fourier Transform:
    $$\hat{{Q}}(\rho)=\sum_g Q(g)\rho(g).$$
As usual, Fourier transforms take convolution into products: $\widehat{{Q}*{Q}}(\rho)=(\hat{{Q}}(\rho))^2$.  For the uniform distribution $U(g)=\frac{1}{|G|}$, Schur's lemma implies that $\hat{U}(\rho)=0$ for all nontrivial $\rho\in \hat{G}$.  The Fourier Inversion Theorem shows how to reconstruct $Q$ from $\hat{Q}(\rho)$:
$${Q}(g)=\frac{1}{|G|}\sum_{\rho \in \hat{G}}d_\rho \operatorname{tr}(\hat{{Q}}(\rho){\rho(g)}^*).$$ The basic Upper Bound Lemma ( see \cite{DiaconisA} page 24) gives a bound on the distance to stationary using the Fourier Transform:
$$4\|Q^{*k}-U\|^2_{TV}\leq \sum_{\substack{\rho\in \hat{G}\\\rho \neq 1}} d_{\rho}\|\hat{Q}^k(\rho)\|^2.$$
In these sums, $d_\rho$ is the dimension of $\rho$, $\|M\|^2=\operatorname{tr}(MM^*)$ and $\rho=1$ is the trivial representation.

    \subsection{ Representation Theory of $H(n)$}\label{PS3.2}
      A neat elementary description of the representations of $H(n)$ appears in \cite{GH}.  
      We find the following self contained derivation more useful for our application.  
      For every divisor $m$ of $n$, $1\leq m\leq n$ , there are $\left(\frac{n}{m}\right)^2 \phi(m)$ distinct irreducible representations of degree $m$, where $\phi(m)$ is the Euler phi function.  
      Since
    $$\sum_{m|n}\left(\frac{n}{m}\right)^2\phi(m)m^2=n^3,$$ this is a complete set.
    
    For the construction, fix $m|n$.  Let 
    $$V=\left\{f:\mathbb{Z}/{m\mathbb{Z}}\rightarrow\mathbb{C}\right\}.$$
    Now fix $a,b\in \mathbb{Z}/(\frac{n}{m}\mathbb{Z})$ and $c\in \mathbb{Z}/m\mathbb{Z}$ with $(c,m)=1$.
    Let $q_m=e^{2\pi i/m}$, $q_n=e^{2\pi i/n}.$  Define $\rho_{a,b,c}(x,y,z):V\rightarrow V$ by $$\rho_{a,b,c}(x,y,z)f(j)=q_n^{ax+by}q_m^{c(yj+z)}f(j+x).$$
  In these formulas, regard $\mathbb{Z}/(\frac{n}{m}\mathbb{Z})$ and $\mathbb{Z}/({m}\mathbb{Z})$ as subgroups of $\mathbb{Z}/n\mathbb{Z}$ by sending $1$ to $m$ (resp. $n/m$.)  Using  $(x,y,z)(x'y'z')=(x+x',y+y',z+z'+xy')$, it is easy to check that $\rho_{a,b,c}$ is a representation:
  \begin{align*}\rho_{a,b,c}(x,y,z)\rho_{a,b,c}(x',y',z')f(j)&=\rho_{a,b,c}(x,y,z)q_n^{ax'+by'}q_m^{c(y'j+z')}f(j+x')\\&=q_n^{a(x+x')+b(y+y')}q_m^{c(y'(j+x)+z'+yj+z)}f(j+x+x')
 \\&=\rho_{a,b,c}(x+x',y+y',z+z'+xy')f(j).
   \end{align*}
   To compute the character, note that the trace of $\rho_{a,b,c}(x,y,z)$ is zero unless $m|x.$  In this case, the trace is $q_n^{ax+by}q_m^{cz}\sum_{j=0}^{m-1}q_m^{cyj}$.  This sum is zero unless $m|y$ (when it is $m$).  Thus:
   $$\chi_{a,b,c}(x,y,z)=\begin{cases}0 &\text{ unless } m|x\text{ and }m|y\\
   q_n^{ax+by}q_m^{cz}m&\text{ if }m|x,y\end{cases}$$
   Distinctness and irreducibility can be determined by computing:
   $$\langle \chi_{a,b,c}|\chi_{a',b',c'}\rangle=\frac{m^2}{n^3}\sum_{j_1,j_2,z}q_n^{m(a-a'){j_1}+m(b-b')j_2}q_m^{(c-c')z}=\delta_{a b c,a' b' c'}.$$  In the sum, $j_1, j_2$ run over $0,1,\dots,n/m-1$ and $z$ runs over $0,1,\dots,n-1$.  Since $a,b$ can take any values in $0,1,\dots, \frac{n}{m}-1$ and $c$ runs over numbers relatively prime to $m$, this makes $\left(\frac{n}{m}\right)^2\phi(m)$ distinct irreducible characters of dimension $m$.  Summing over divisors of $n$ accounts for all of the characters.  
   \begin{example}When $m=1$, this gives the $n^2$ one dimensional characters $\chi_{a,b}(x,y,z)=q_n^{ax+by}$.  Note that $q_1=1$ so that $c$ doesn't occur.)
   \end{example}
   \begin{example}  When $m=2$, set $$\delta(x)=\begin{cases}1 &\text{if $x$ is even}\\0 &\text{ if $x$ is odd.},\end{cases}$$ with   $ \overline{\delta(x)}=1-\delta({x})$. 
   $$\rho_{a,b,1}(x,y,z)=q_n^{ax+by}q^z_m\begin{bmatrix}\delta(x)&
\overline{\delta(x)}\\   \overline{\delta(x)}q^y_m&\delta(x)q^y_m\end{bmatrix}$$   Note that $\operatorname{tr}(\rho_{a,b,1}(x,y,z))$ agrees with $\chi_{a,b,1}$ above.
   \end{example}
   \begin{example}Consider $\hat{   \mathcal{Q}}(a,b,c)$ for $$\mathcal{Q}(1, 0, 0 )=\mathcal{Q}(-1 ,0 ,0)=\mathcal{Q}(0 ,1 ,0)=\mathcal{Q}(0 ,-1 ,0)=\frac{1}{4}.$$
  When $m=1$, $\hat{\mathcal{Q}}(a,b)=\frac{1}{2}\cos\left(\frac{2\pi a}{n}\right)+\frac{1}{2}\cos\left(\frac{2\pi b}{n}\right)$ for $0\leq a,b\leq n-1$.
  
  When $m=2$ (so $n$ is even) $$\hat{Q}(a,b,1)=\frac{1}{2}\begin{bmatrix}\cos\left(\frac{2\pi b}{n}\right)&\cos\left(\frac{2\pi a}{n}\right)\\\cos\left(\frac{2\pi a}{n}\right)&-\cos\left(\frac{2\pi b}{n}\right)\end{bmatrix}.$$
  
  In general, when $m\geq 3$, 
$$\hat{\mathcal{Q}}(a,b,c)=\frac{1}{4}\begin{tikzpicture}[baseline=(current bounding box.center)]
\matrix (m) [matrix of math nodes,nodes in empty cells,right delimiter={)},left delimiter={(} ]{
\phantom{q_n^{-a}}  & q_n^{a} &  & &&   & q_n^{-a}  \\
 q_n^{-a}& & & & &&   \\
  & & & & &&   \\
  & & q_n^{-a}&\phantom{q_n^{-a}} &q_n^{a}& &    \\
   & & & & &&   \\
  & & & & &&q_n^{a}  \\
q_n^{a} & & && &  q_n^{-a}&\phantom{q_n^{-a}} \\
} ;
\draw[loosely dotted,thick] (m-1-1.center)-- (m-4-4);
\draw[loosely dotted,thick] (m-2-1)-- (m-4-3);
\draw[loosely dotted,thick] (m-1-2)-- (m-4-5);
\draw[loosely dotted,thick] (m-4-4)-- (m-7-7.center);
\draw[loosely dotted,thick] (m-4-3)-- (m-7-6);
\draw[loosely dotted,thick] (m-4-5)-- (m-6-7.west);
\node[align=center] (cosgen) at (5.3cm,1.4cm) [below=3mm]{$q_n^bq_m^{jc}+q_n^{-b}q_m^{-jc}$\,\,\\\hspace{1cm}\tiny{$0\leq j \leq m-1$}};
\path[thick, bend left=35, <-] 
 (m-4-4.center) edge (cosgen.west);
\end{tikzpicture}$$  
Note that conjugating $\hat{Q}(a,b,c)$ by a diagonal matrix with $1,q_n^a,q_n^{2a},\dots q_n^{(m-1)a}$ on the diagonal results in a matrix with ones on the super and subdiagonals (and $q_n^{-ma}$, $q_n^{ma}$ in the two corners).
   \end{example}
   \begin{example}  When $n$ is a prime $p$ , there are $p^2$ one-dimensional representations and $p-1$ $p$-dimensional representations, with $\hat{Q}(0,0,c)$ conjugate to $M(c)$ of the introduction.
   \end{example}
            \subsection{ Proof of Theorem \ref{PT1} when $n=p$ is prime}\label{PS3.3}
            From the Upper Bound Lemma stated in Section \ref{PS3.1} and the formulas for $\hat{\mathcal{Q}}(\rho)$ in Section \ref{PS3.2},
            \begin{align*}4\|Q^{*k}-U\|^2_{TV}\leq\sum_{(a,b)\neq (0,0)}\left(\frac{1}{2}\cos\left(\frac{2\pi a}{p}\right)+\frac{1}{2}\cos\left(\frac{2\pi b}{p}\right)\right)^{2k}\\\vspace{3cm}+\sum_{\xi=1}^{p-1}p\|M(\xi)^k\|^2=I+\mathit{II}\end{align*}
            To understand $I$, first consider  $a=0$ and $b=1$.  This term is $$\left(\frac{1}{2}+\frac{1}{2}\cos\left(\frac{2\pi}{p}\right)\right)^{2k}=\left(1-\frac{\pi^2}{p^2}+O\left(\frac{1}{p^4}\right)\right)^{2k}.$$  This term is of size $e^{-B\eta}$ if $k=\eta p^2$ with $B=2\pi^2$.   A standard analysis, carried out in detail in \cite{DiaconisA} shows that the other terms in (I) can be similarly bounded and that the sum of these terms is at most $Ce^{-4\pi^2\eta}$ for a universal constant $C$, independent of $p$ and $k$.
            Consider the next sum ($\mathit{II}$).  The matrix $M(\xi)$ is symmetric with real eigenvalues $1>\beta_1(\xi)\geq\dots \geq \beta_p(\xi)>-1$. In Section \ref{PS4} it is shown that for $0<\xi<\frac{p}{\log(p)}$, $$\beta^*(\xi)=\max(\beta_1(\xi),-\beta_p(\xi))\leq 1-\frac{\theta}{p^{4/3}}$$  for a universal constant $\theta$.  Also, for $\frac{p}{\log(p)}\leq\xi<\frac{p}{2}$, it is shown that
            $$\beta^*(\xi)\leq 1-\frac{3}{4}\left(\frac{\xi}{p}\right)^2.$$  Furthermore note that $$\beta^*(\xi)=\beta^*(p-\xi).$$
              Indeed, as mentioned in Section \ref{PS2.5}, much sharper bounds are found.  These are not relevant for this proof.  Using these bounds with $\overline{\theta}=\min\{\theta,\frac{3}{4}\}$, $$\mathit{II}\leq p^3\left(1-\frac{\overline{\theta}}{p^{4/3}}\right)^{2k}+p^3\left(1-\frac{\overline{\theta}}{\log(p)}\right)^{2k}.$$
            This is exponentially smaller than the bound for $I$, completing the proof of the upper bound.
            
            A lower bound for the total variation follows from a lower bound for the total variation convergence of the $(1,2)$ coordinate $X_k$ to the uniform distribution on $\mathbb{Z}/p\mathbb{Z}$.  As explained in Section \ref{PS2.3}, $X_k$ evolves as a simple random walk on $\mathbb{Z}/p\mathbb{Z}$ and a lower bound showing that $k=\theta p^2$ steps is not sufficient for a fixed $\theta$ is well known.  See \cite{DiaconisB} for several different proofs.  This completes the proof.  
            
            We will not carry out the details of the proof of Theorem \ref{PT1} in the case of general $n$.  Bounds for the various transforms $\hat{\mathcal{Q}}(a,b,c)$ are derived in Section \ref{PS4}.  No new ideas are needed.
  \subsection{ A Formula for the Center} \label{PS3.4}  Let $n=p$ be an odd prime.  Let $Z_k$ be the value of the central element after $k$ steps of the walk.  A formula for $P\{Z_k=z\}$ can be derived from the Fourier Inversion Theorem.  We do not see how to use this to derive useful estimates but record it here in case someone else does.  
  \begin{proposition} For a simple random walk on $H(p)$ with $p$ an odd prime, all $k\geq0$, and $z \in \mathbb{Z}/p\mathbb{Z}$,
  $$P\{Z_k=z\}=\frac{1}{p}+\frac{1}{p^2}\sum_{\xi=1}^{p-1}e^{\frac{-2\pi i \xi z}{p}}\sum_{j=1}^{p}(\hat{\mathcal{Q}}(\rho_\xi)^{k})_{1,j}.$$
  \end{proposition}
\begin{proof}  From the Fourier Inversion Theorem, for any $(x,y,z)$ 
$$P(X_k=x,Y_k=y, Z_k=z)=\frac{1}{p^3}\sum_{\rho\in \hat{G}}d_\rho\operatorname{Tr}\left((\hat{{Q}}(\rho))^k\rho((x,y,z)^{-1})\right).$$
Sum over $x,y$ and bring this sum inside the trace.  The terms for the 1-dimensional representations are easily seen to vanish except of course for the trivial representation which contributes $p^2/p^3=1/p$.  For the $p$ dimensional representations, writing $\rho_\xi$ for $\rho_{0,0,\xi}$, $\rho_\xi((x,y,z)^{-1})=\rho_\xi(-x,-y,-z+xy)$.  In the basis of Section \ref{PS2.3}
$$\rho_\xi(-x,-y,-z+xy)f(j)=e^{\frac{2\pi i \xi}{p}(-yj-z+xy)}f(j-x).$$
Fixing $x$ and summing over $y$, $\sum_y e^{\frac{-2\pi i \xi}{p}(j-x)y}=p \delta(j,x) $.  Then,  $$\sum_{y}\rho_\xi(-x,-y,-z+xy)f(j)=pe^{\frac{-2\pi i \xi z}{p}}f(0).$$  Summing over $x$ gives $$\sum_{x,y}\rho_\xi(-x,-y,-z+xy)=pe^{-\frac{2\pi i \xi z}{p}}\begin{bmatrix}1&0&\dots&0\\\vdots&\vdots&&\vdots
\\1&0&\dots&0\end{bmatrix}.$$  Multiplying by $\hat{{Q}}(\rho_\xi)^k$ and taking the trace gives the result.
\end{proof}
\begin{remark}  Of course, the elements in the first row of $\hat{{Q}}^k(\rho_\xi)$ tend to zero, but it does not seem simple to get quantitative bounds.
\end{remark}   
    \section{Eigenvalue Bounds} \label{PS4}  This section gives upper and lower bounds for the eigenvalues of the matrices $M(\xi)$ of (\ref{ourmatrix}}).  Recall that these are $n\times n$, with $\frac{1}{4}$ in the upper right and lower left corners and $\frac{1}{2}\cos(\frac{2\pi \xi j}{n})$ for $0\leq j\leq n-1$ on the main diagonal.  These matrices have some symmetry properties, noted in Section \ref{PS4.1}, which simplifies the job.  Section \ref{PS4.2} gives upper bounds for the largest eigenvalue of the form $1-\theta\left(\frac{\xi}{n}\right)^{\frac{4}{3}}$ for $\theta$ an absolute positive constant and $\xi<<n$.  The argument makes novel use of bounds for Dirichlet eigenvalues using geometric path techniques.  It is useful for more general diagonal entries and less structured matrices.  A different variant bounds the largest eigenvalues for larger $\xi$.  Section \ref{PS4.3} gives lower bounds for the smallest eigenvalues.  These use quite different ideas:  Comparisons with matrices constructed to have `nice' eigenvalues.
    
    The following series of figures shows numerical examples.  Figure \ref{PF1} shows the top eigenvalues $\beta_1(\xi)$ of $M(\xi)$ when $n=150$.  As $\xi$ varies from $1,2,\dots$ for small $\xi$, the top eigenvalues are close to 1 and fall off as $\xi$ increases but they are not monotone.  The evident symmetry $\beta_1(c)=\beta_1(n-c)$ is proved in Section \ref{PS4.1}.  Figure \ref{PF2} shows the smallest eigenvalues of $M(\xi)$.  While these appear as a mirror image of Figure \ref{PF1}, a closer look at the numbers shows this is only an approximation and the same proof method does not work.  Figure \ref{PF3} shows all of the eigenvalues of M(1) for $n=150$.  
\begin{figure}
\includegraphics[width=8cm]{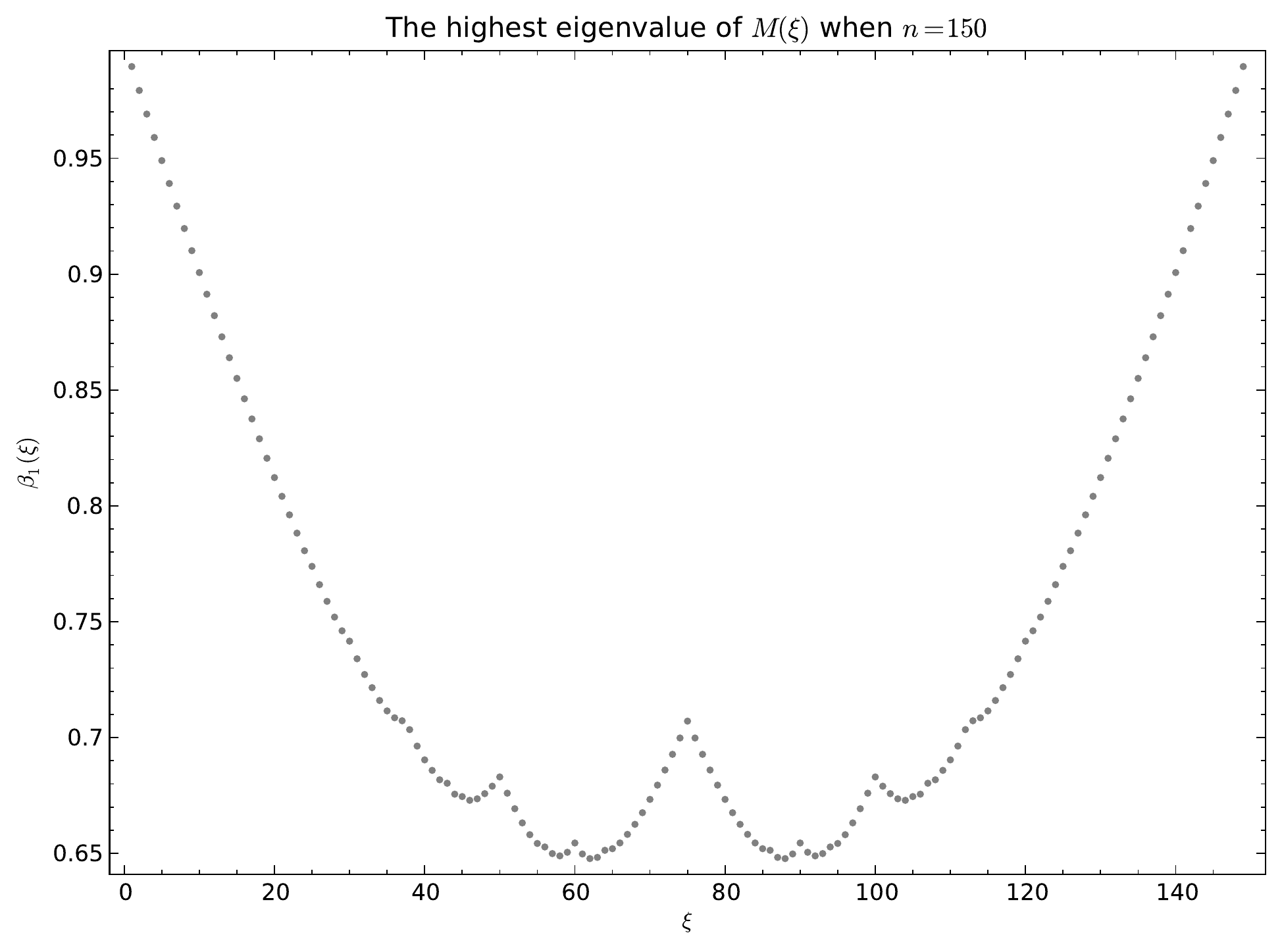}
\caption{The top eigenvalues of $M(\xi)$}\label{PF1}
\end{figure}
\begin{figure}
\includegraphics[width=8cm]{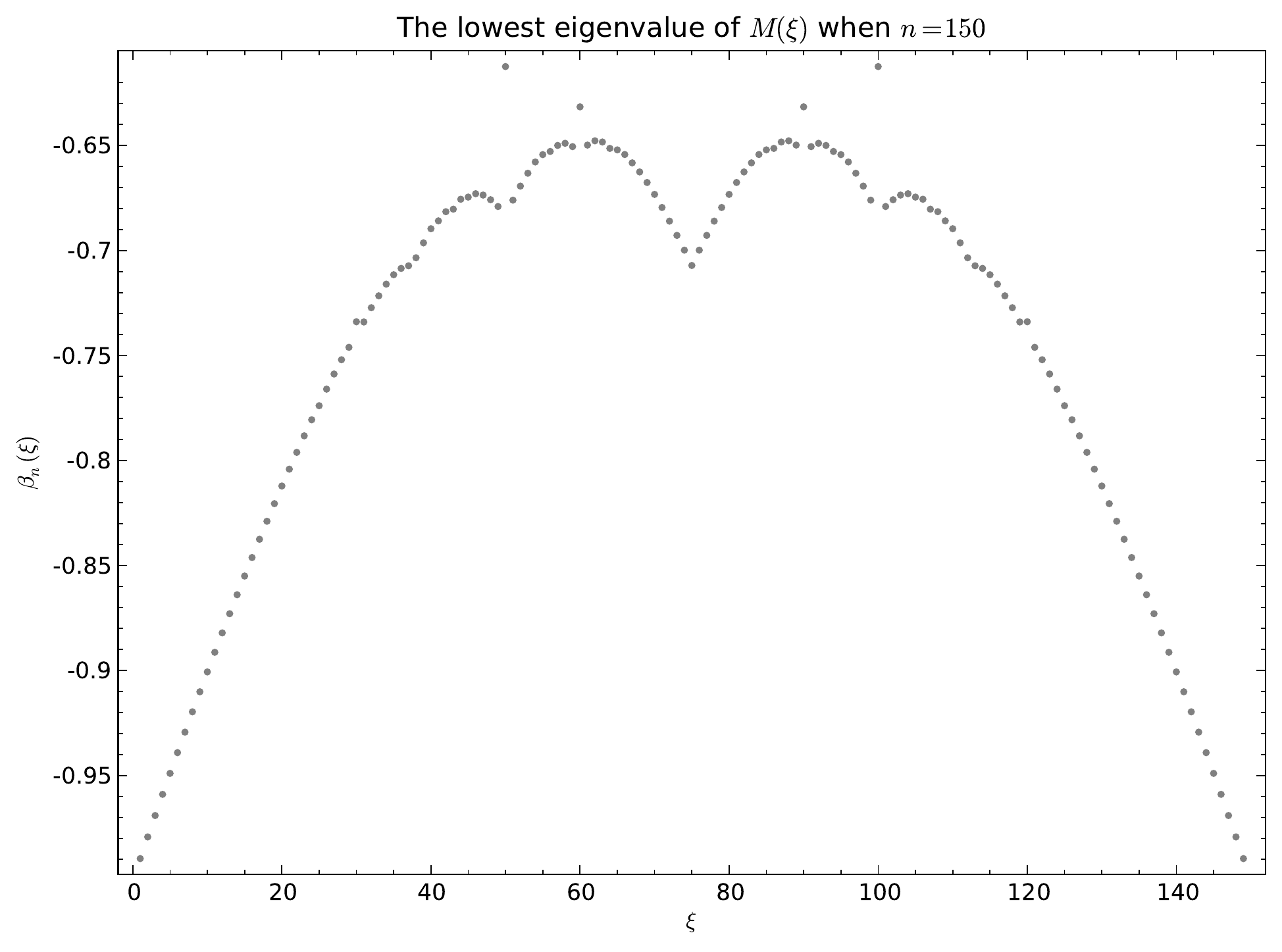}
\caption{The lowest eigenvalues of $M(\xi)$}\label{PF2}
\end{figure}
\begin{figure}
\includegraphics[width=8cm]{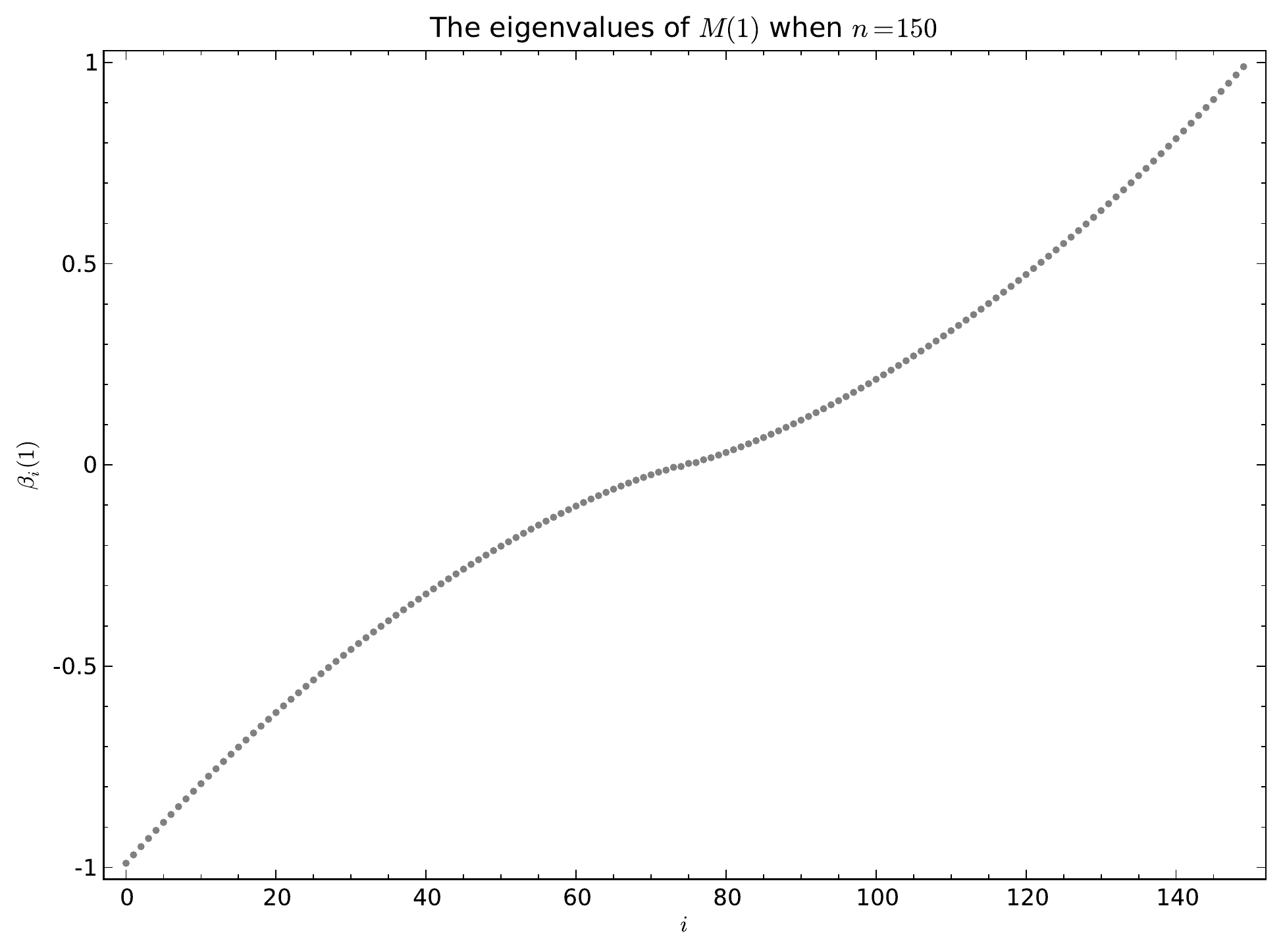}
\caption{All the eigenvalues of $M(1)$}\label{PF3}
\end{figure}
    \subsection{Symmetries and Inclusions} \label{PS4.1}  While many of the properties proved in this section are quite robust, the following more special properties are needed.  For positive integers $n$, $\xi$ and any real $\alpha$, let $D_{n,\xi,\alpha}$ denote the $n\times n$ diagonal matrix with diagonal entries:
     $$D^{(j)}_{n,\xi,\alpha}=\cos\left(\frac{2\pi \xi}{n}(\alpha+j)\right)\,\,\, 0\leq j \leq n-1.$$
     Let $M_{n,\xi,\alpha}=\frac{1}{2}(D_{n,\xi,\alpha}+\mathcal{P})$) where $\mathcal{P}$ is the transition matrix of a simple random walk on $\mathbb{Z}/n\mathbb{Z}$ ($\frac{1}{2}$ on the diagonals just above and below the main diagonal and $\frac{1}{2}$ in the upper right and lower left corners.)  Thus $M_{n,\xi,0}=M(\xi)$.  Let $S(n,\xi,\alpha)$ be the spectrum of $M_{n,\xi,\alpha}$.
     \begin{proposition}\label{PP4.1} \begin{enumerate}[(a)]
     \item\label{PIA} $S(n,\xi,0)=S(n,n-\xi,0)$
     \item \label{PIB} For all positive integer $k$, $S(n,\xi,\alpha)\subseteq S(kn,k\xi,\alpha)$
     \item \label{PIC} If $n$ is even, $S(n,\xi, \alpha)=-S\left(n,\xi,\alpha +\frac{n}{2\xi}\right)$
     \item \label{PID}  If $n$ is odd, $S(n,\xi, \alpha)\subseteq -S\left(2n,2\xi,\alpha +\frac{n}{2\xi}\right)$
     \end{enumerate}
     \end{proposition}
     \begin{proof}  For (\ref{PIA}), $$\cos\left(\frac{2\pi (n-\xi)}{n}j\right)=\cos\left(2\pi j-\frac{2\pi \xi}{n}j\right)=\cos\left(\frac{2\pi \xi}{n}j\right).$$  Thus $M(n,\xi,0)=M(n,n-\xi,0)$.
     For (\ref{PIB}), if $\phi$ is an eigenvector of $M_{n,\xi,\alpha}$ then juxtaposing it $k$ times gives an eigenvector of $M_{kn,k\xi,\alpha}$ with the same eigenvalue because of periodicity.
For (\ref{PIC}), suppose $\psi$ is an eigenvector of $M_{n,\xi,\alpha}$ with associated eigenvalue $\theta$.  Thus:
$$\frac{1}{4}(\psi(j+1)+\psi(j-1))+1/2\cos\left(\frac{2\pi \xi}{n}(\alpha+j)\right)\psi(j)=\theta\psi(j)$$
Set $\phi(j)=(-1)^j\psi(j)$.  Then
$$\frac{1}{4}(\phi(j+1)+\phi(j-1))-\frac{1}{2}\cos\left(\frac{2\pi \xi}{n}(\alpha+j)\right)\phi(j)=-\theta \psi(j)$$
but $-\cos(\frac{2\pi \xi}{n}(\alpha+j)=\cos\left(\frac{2\pi \xi}{n}\left((\alpha+j+\frac{n}{2\xi}\right)\right)$.

For (\ref{PID}), properties (\ref{PIB}) and (\ref{PIC})  imply that $$S(n,\xi,\alpha)\subseteq S(2n,2\xi,\alpha)=-S\left(2n,2\xi,\alpha+\frac{n}{2\xi}\right).$$
     \end{proof}
\begin{remark}
  The operator with 1s above and below the diagonal (and in the corners) and $V\cos(\frac{2\pi \xi}{n}(j+\alpha))$ with $V>0$ a constant is a discrete version of the ``almost Mathieu operator'' of solid state physics. See \cite{ZMK}.    The discrete operator corresponds to $4M(n,\xi,\alpha)$ when $V=2$.  This is well known to be the critical case in the theory.  The $\xi$'s of physical interest scale as $O(n)$, which corresponds to the case when $\mathbb{Z}_n$ is replaced by $\mathbb{Z}$. When $V>2$ the eigenfunctions are localized.  For $V<2$, they are spread out.  The spectrum when $V=2$ is called ``Hofstadter's Butterfly.''  Its spectral properties have recently been observed in graphene \cite{Petal}.  All our $\frac{\xi}{n}$ ratios are rational.   It is a fascinating question if the limit as $\frac{\xi}{n}$ tends to an irrational has limiting spectrum matching the spectrum of Harper's operator on $\ell^2$.  See \cite{ZMK}  for further discussion.
  \end{remark}
    \subsection{Dirichlet Path Arguments} \label{PS4.2}  For general $n$, $M(\xi)=\hat{Q}(0,0,\xi)$ is an $n\times n$ matrix with $\frac{1}{2}\cos\left(\frac{2\pi \xi j}{n}\right)$ on the main diagonal, $\frac{1}{4}$ on the super and subdiagonals and in the upper right and lower left corners.      Then $\frac{1}{2}I+M(\xi)$ has non-negative entries with row sums in $[\frac{1}{2},\frac{3}{2}]$.  
    Finally, $L_\xi=\frac{1}{3}I+\frac{2}{3}M(\xi)$ is a substochastic matrix with $\frac{1}{6}$ on the super and subdiagonals (and the corners) and $\frac{1}{3}\left(1+\cos\left(\frac{2\pi \xi j}{n}\right)\right)$ on the main diagonal for $0\leq j\leq n-1$.
    The row sums are between $\frac{1}{3}$ and 1.  Form an $(n+1)\times(n+1)$ stochastic matrix $K_\xi$ by adding an absorbing state $\infty$,     with transitions from $j$ to $\infty$,
     $K_\xi(j,\infty)=\frac{1}{3}\left(1-\cos\left(\frac{2\pi \xi j}{n}\right)\right)$. Set $K(\infty,\infty)=1$.  Bounds on the eigenvalues of $L_\xi$ (and so $M(\xi)$) give bounds on the Dirichlet eigenvalues of $K_\xi$.
     Use the notation of \cite{DM} Section 4.2 with $S=\{0,1,\dots,n-1\}$, $\overline{S}=S\cup \left\{\infty\right\}$.  The Markov chain $K_\xi$ restricted to $S$ is connected.  Let $U(j)=1/n$, $0\leq j\leq n-1$ with 
     $$\ell^2(U)=\left\{f:\{0,1,\dots,n-1\}\rightarrow \mathbb{R}\right\}$$ 
having inner product $\langle f|g \rangle_U=\sum_{j\in S} f(j)g(j)U(j)$.  As required, for all $x,y \in S$, $U(x)K_\xi(x,y)= U(y)K_\xi(y,x)$.  When needed, functions $f$ on $S$ are extended to $\overline{S}$ as $f(\infty)=0$. (Also set $U(\infty)=0$.)  Define a Dirichlet form $\mathcal{E}$ on $\ell^2(U)$ by
$$\mathcal{E}(f|f)=\frac{1}{2}\sum_{x,y\in \overline{S}} (f(x)-f(y))^2U(x)K_\xi(x,y).$$
Let $\beta=\beta(\xi)$ be the largest eigenvalue of $L_\xi$.  From \cite{DM} Lemma 19, if there is a constant $A> 0$ such that $\|f\|_U^2\leq A\|\mathcal{E}(f)\|$, for all $f\in \ell^2(U)$, then $\beta\leq 1-\frac{1}{A}$.  This shows that the largest eigenvalue of $M(\xi)$ is at most $1-\frac{3}{2 A}$.

Geometric arguments require a choice of paths $\gamma_x$ starting at $x\in S$ to $\infty$, with steps which are possible with respect to $K_\xi$.  Thus
  $\gamma_x=(x_0=x,x_1,x_2,\dots,x_d=\infty)$ with $K_\xi(x_i,x_{i+1})>0$ for $0\leq i\leq d-1$.  Let $d=|\gamma_x|$ denote the length of this path.  From \cite{DM} Proposition 2c,
  \begin{align}\label{Aeq2}
  \beta\leq 1- \frac{1}{A} \text{ with } A=\max_{x\in S,y\in \overline{S}} \frac{2}{K_\xi(x,y)}\sum_{z\in S,\,(x,y)\in \gamma_z}|\gamma_z|.
  \end{align}
  
  The bounds are better if $A$ is `small.'  From (\ref{Aeq2}) this happens if both the paths are chosen so that no edge occurs too often and the $(x,y)$ edges that appear don't have small value of $K_\xi(x,y)$.
  With these preliminaries we next prove:
  \begin{proposition}\label{PP1}There exists a positive constant $\theta$ such that for all $n$ and $1\leq\xi\leq \frac{n}{\log{n}}$, the largest eigenvalue $\beta(\xi)$ of $M(\xi)$ satisfies $\beta(\xi)\leq 1-\theta\left(\frac{\xi}{n}\right)^{4/3}$.
  \end{proposition}
  \begin{figure}
\begin{center}
\includegraphics[width=9cm]{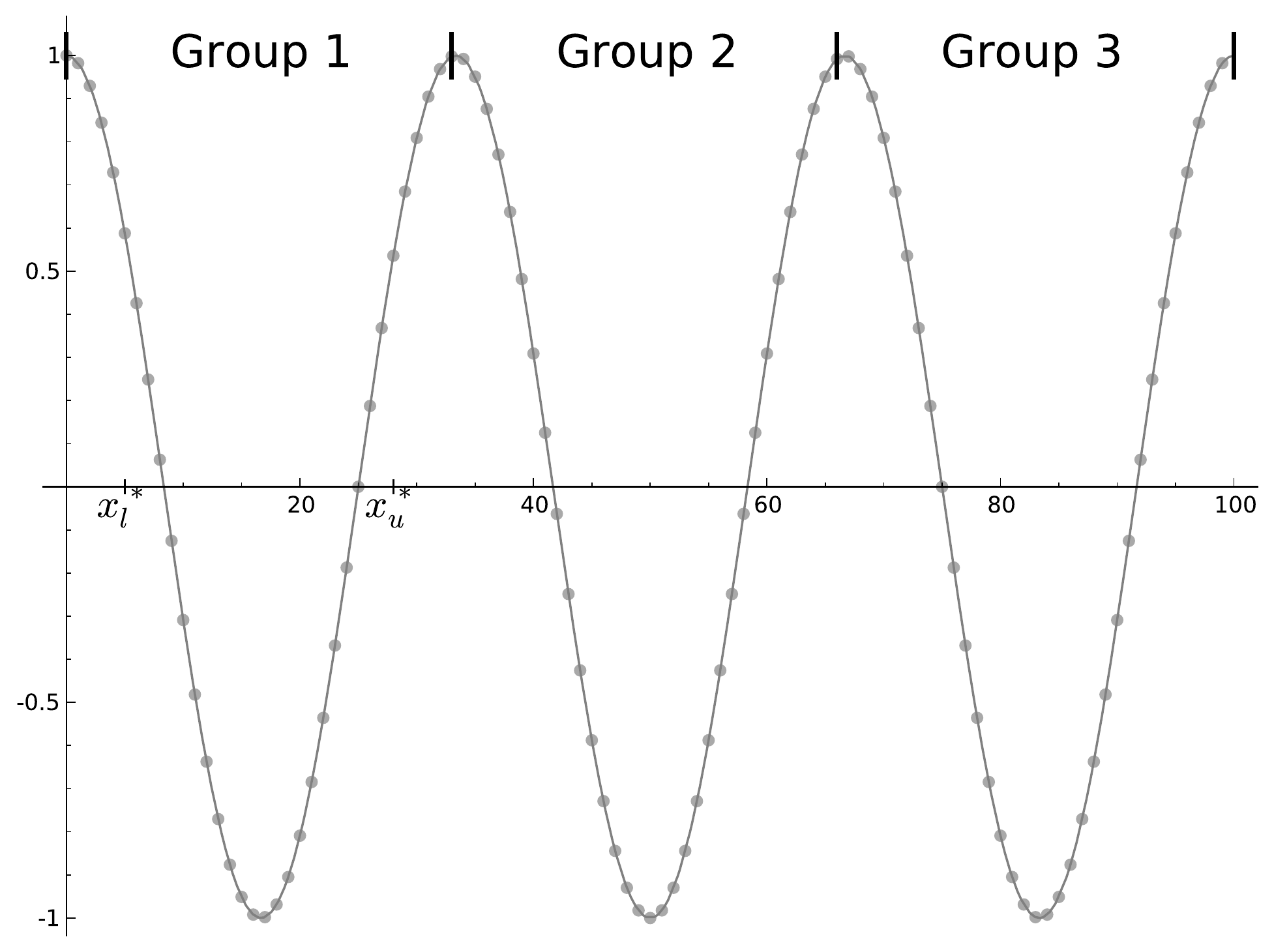}
\end{center}
\caption{Splitting points into groups according to the cycles of cosine.  Note the values marked $x_l^*$ and $x_u^*$.}\label{Fig:cos}
\end{figure}
 \begin{proposition}\label{PP2} For $\frac{n}{\log(n)}\leq \xi\leq \frac{n}{2}$,
 $$\beta(\xi)\leq 1-\frac{3}{4}\left(\frac{\xi}{n}\right)^2.$$
 \end{proposition}
 \begin{proof}[Proof of Proposition \ref{PP1}]
For $1\leq \xi\leq \frac{n}{\log{n}}$, let $x^*=\left\lfloor \left(\frac{n}{\xi}\right)^{\frac{2}{3}}\right\rfloor$.  Break the $\xi$ cycles of $\cos\left(\frac{2\pi \xi j}{n}\right)$ at $\lfloor \frac{n}{\xi}\rfloor,2\lfloor\frac{n}{\xi}\rfloor,\dots$ into $\xi$ groups, eg.\ when $\xi=3$  breaking ties to the right.  See Figure \ref{Fig:cos}.
Each group has $\frac{n}{\xi}+O(1)$ points.  For each group, define a point $x_l^*$ by counting $x^*$ points from the left and $x_u^*$ by counting $x^*$ from the right of the groups.  Each group is treated identically and we here focus on the first group.  From the development above, for any $j$ between $x_l^*$ and $2 x_l^*$,
$$K_\xi(j,\infty)=\frac{1}{3}\left(1-\cos\left(\frac{2\pi \xi j}{n}\right)\right)\geq\frac{1}{3}\left(\frac{1}{2}\left(\frac{2\pi \xi x_l^*}{n}\right)^2+O\left(\frac{2\pi \xi x_l^*}{n}\right)^4\right)$$
Thus $\frac{2}{K_\xi(j,\infty)}\leq\theta\left(\left(\frac{n}{\xi x_l^*}\right)^2\right)(1+o(1))$ for an explicit constant $\theta$, independent of $\xi, j, n$.  For points $r$ between $1$ and $x_l^*-1$, connect $r$ to infinity by horizontal paths  $r, r+1,\dots,(r+x_l^*+1)$ going directly from there to $\infty$. See Figure \ref{AF2B}.
\tikzstyle{vertex}=[circle,fill=black,inner sep=0,minimum size=2mm]
\begin{figure}
\begin{tikzpicture}
\node[vertex] (vert_1) at (0,0) {};
\node[vertex] (vert_2) at (1,0) {} edge (vert_1);
\node[vertex] (vert_3) at (2,0) {};
\draw (vert_2) edge (vert_3);
\draw[dashed,thick] (vert_2.330) edge (vert_3.210);
\node[vertex] (vert_xm1) at (3,0) {} edge[loosely dotted] (vert_3);
\node[vertex] (vert_x) at (4,0) {} edge (vert_xm1);
\draw[dashed,thick] (vert_x.330) edge (vert_xm1.210);
\node[vertex] (vert_xp1) at (5,0) {} ;
\draw[dashed,thick] (vert_xp1) edge (vert_x);
\node[vertex] (vert_infty) at (5,1) {};
\node [below=.1cm] (1) at (vert_1) {$1$};
\node [below=.1cm] (2) at (vert_2) {$2$};
\node [below=.1cm] (x) at (vert_x) {$x_l^*$};
\node [below=.1cm] (xp1) at (vert_xp1) {$x_{l}^*+1$};
\node [right=.1cm] (infty) at (vert_infty) {$\infty$};
\draw[dashed,thick] (vert_xp1) edge (vert_infty);
\draw[thick] (vert_x) edge (vert_infty);
\end{tikzpicture}
\caption{The chosen path from $1$ to $\infty$ is given with solid lines, while the chosen path from $2$ to $\infty$ is pictured with dashed lines.}\label{AF2B}
\end{figure}
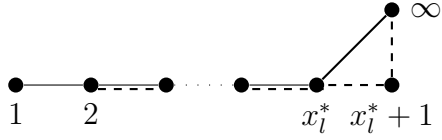

Do the symmetric thing with the rightmost $x^*-1$ points at the right of group 1, (moving leftward.)  Connect points from $x_l^*$ to $x_u^*$ directly to $\infty$ with paths of length 1.  Treat each of the other groups in a parallel fashion.

Bounds for $A$ in (\ref{Aeq2}) follow from bounds from the edges in Group 1.  For the edge $(r-1,r)$, order $x^*$ paths use this edge and each such path has length $x^*$.  Since $\frac{2}{K_\xi(r-1,r)}\leq\frac{2}{K_\xi(x_l^*-1,x_l^*)}=12$, the contribution from this edge is at most of order $(x^*)^2$.  Next consider edges $(j,\infty)$ with $j$ between $x_l^*$ and $2x_l^*$ or between $x_u^*-x^*$ and $x_u^*$.  At most two paths use such an edge, one of length $x^*$ (and the other of length one.)  The contribution here is of order $\left(\frac{n}{\xi x^*}\right)^2 x^*$. The exponent $\frac{2}{3}$ in our choices of $x^*$ was chosen to make $(x^*)^2$ and $\left(\frac{n}{\xi}\right)^2\frac{1}{x^*}$ of the same order.  The only other edges used are $(j,\infty)$ with $j$ from $2x_l^*$ to $2x_u^*$.  These are used once for paths of length 1 and $\frac{1}{K(j,\infty)}$ is at most of order $\left(\frac{n}{\xi}\right)^{2/3}$.  Combining bounds completes the proof of Proposition \ref{PP1}.
 \end{proof}
 \begin{proof}[Proof of Proposition \ref{PP2}]  Proceeding as above, break $0\leq j\leq n-1$ into groups of $\frac{n}{\xi}$ points (within 1).  Again, the groups are treated identically and we focus on the first group.  Let $x_l^*=\frac{1}{4}\left\lfloor\frac{n}{\xi}\right\rfloor$ and connect points of $0\leq j<x_l^*$ to points $x_l^*+j$ and then to $\infty$.  Since $\cos\left(\frac{2\pi}{n}\xi (x_l^*+j)\right)\leq 0$, $\frac{2}{K_\xi(x_l^*+j,\infty)}\leq 6.$  For $x_l^*$ to $x_u^*$, connect points directly to $\infty$.  For $j$ at the right of group 1, connect backward to $x_u^*-j$ and then to $\infty$ in symmetric fashion.  Now, $\frac{2}{K(j,\infty)}\leq 6$.  For all $x_l^*\leq j\leq x_u^*$ that occur in this choice of paths, these edges (from $j$ to $\infty$) have at most 3 paths through them, the longest path of length of order $\frac{n}{4\xi}$.  This gives $5\frac{n}{\xi}$ as an upper bound for those edges.  For edges $(j,j+1)$, $\frac{2}{K(j,j+1)}=12$. The maximizing such edge is $(x_{l}^*-1,x_l^*).$  This has $x_l^*$ paths using it of length $x^*$ contributing at most $12(x^*)^2=\frac{3}{4}\left(\frac{n}{\xi}\right)^2$.  Combining bounds completes the proof of Proposition \ref{PP2}.  
 \end{proof}
    \subsection{Lower bounds on Negative Eigenvalues} \label{PS4.3}  We derive crude but useful lower bounds on the negative eigenvalues of $M(\xi)$ by using the upper bounds from Section \ref{PS4.2} together with the inclusion from Proposition \ref{PP4.1}.
    Suppose that $n$ is odd (the case of interest for proving Theorem \ref{PT1}). From Proposition \ref{PP4.1} (\ref{PID}), the following spectral inclusion holds:
    $$S(n,\xi,0)\subseteq-S(2n,2\xi,\frac{n}{2\xi}).$$  Thus bounds on the smallest eigenvalues of $M(\xi)$ follow from bounds on the largest eigenvalues of $2n\times 2n$ matrices with diagonal $\cos\left(\frac{2\pi\xi j}{n}+\pi\right)$, $0\leq j\leq 2n-1$. First consider the case of $\xi=1$, the diagonal elements appear as in Figure \ref{AF1}.
\begin{figure}
\begin{center}
 \includegraphics[width=9cm]{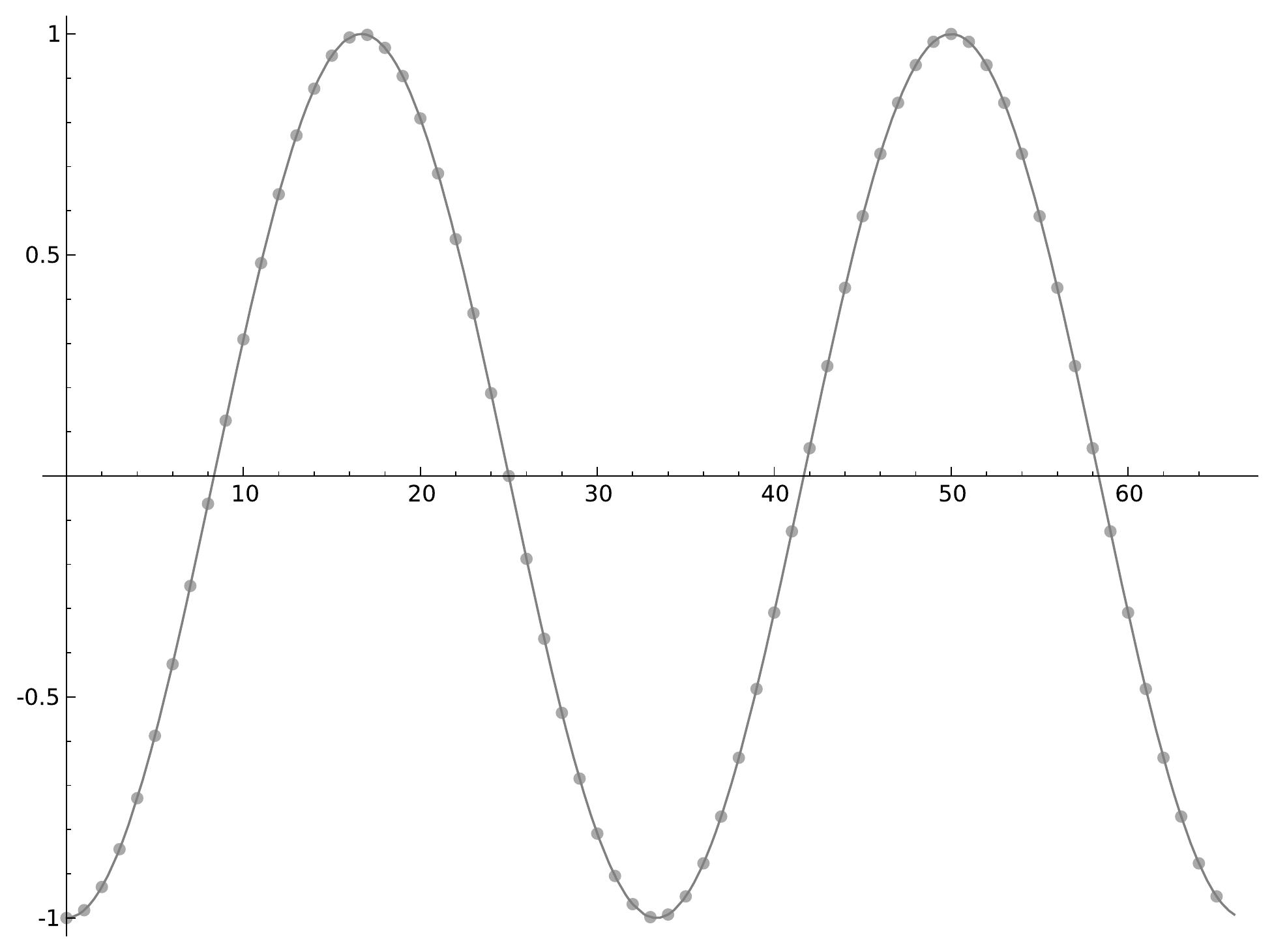}
\end{center}
\caption{Diagonal elements for $\xi=1$}\label{AF1}
\end{figure}
  A cyclic shift of the diagonal doesn't change the spectrum so we may consider the $2n\times 2n$ matrix with diagonal entries $\cos\left(\frac{\pi(2j+1)}{n}\right),\,\,0\leq j<2n-1$, as in Figure \ref{AF4}. 
\begin{figure}
\begin{center}
 \includegraphics[width=9cm]{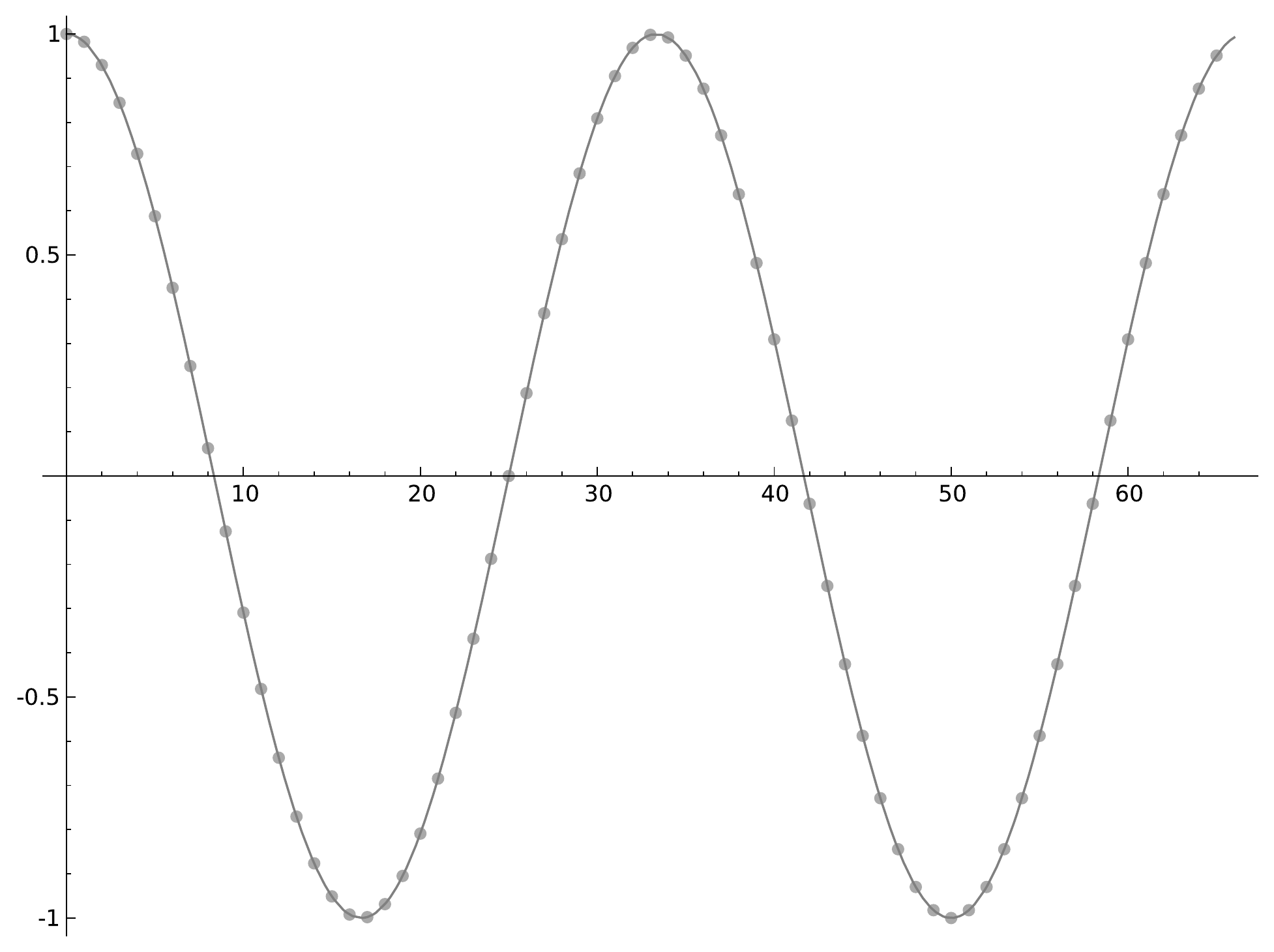}
\end{center}
\caption{A shift of the diagonal elements for $\xi=1$}\label{AF4}
\end{figure}
Now, the same paths and bounds used in Section \ref{PS4.2} can be used to get upper bounds.  We will not repeat the details but merely state the conclusions.
\begin{proposition}  There is a positive constant $\theta$ such that for all odd $n$ and $1\leq \xi\leq \frac{n}{\log(n)}$, the smallest eigenvalue $\beta_*(\xi)$ of $M(\xi)$ satisfies
$$\beta_*(\xi)\geq -1+O\left(\frac{\xi}{n}\right)^{4/3}.$$
\end{proposition}
\begin{proposition}For $\frac{n}{\log({n})}\leq \xi<\frac{n}{2}$, $\beta_*(\xi)\geq -1+\frac{3}{4}\left(\frac{\xi}{n}\right)^2(1+o(1))$
\end{proposition}
\begin{remark}  The case of even $n$ arises naturally as well.  This may be treated in parallel fashion using Proposition \ref{PP4.1} (\ref{PIC}) to reduce things to bound the largest eigenvalues of slightly shifted matrices as above.
\end{remark}      

\bibliographystyle{plain}
\bibliography{Heisenbergbib}
\end{document}